\documentclass[12pt]{amsart}

\usepackage{times}
\usepackage{amsmath,amsfonts,amssymb,amsopn,amscd,amsthm}
\usepackage{mathbbol}
\usepackage{mathrsfs}
\usepackage{graphicx,xspace}
\usepackage{epsfig}
\usepackage{enumerate} 
\usepackage{enumitem}
\usepackage{xcolor}
\usepackage[all]{xypic}

\usepackage[T1]{fontenc}
\usepackage[utf8]{inputenc}

\usepackage{hyperref}
\usepackage{psfrag}
\usepackage{bm}
\usepackage{youngtab}
\usepackage{tikz}
\usetikzlibrary{snakes}
\usepackage[all]{xy}
\usepackage{varioref}

\theoremstyle{plain}
\newtheorem{theoreme}{Theorem}[section]
\newtheorem{prop}[theoreme]{Proposition}

\newtheorem{cor}[theoreme]{Corollary}

\theoremstyle{remark}
\newtheorem{definition}[theoreme]{Definition}
\newtheorem{ex}[theoreme]{Example}
\newtheorem{remark}[theoreme]{Remark}
\newtheorem{c-ex}[theoreme]{Counter-example}

\date{}

\DeclareUnicodeCharacter{03C7}{\chi}
\DeclareUnicodeCharacter{B2}{^{2}}
\DeclareMathOperator{\supp}{supp}

\DeclareMathOperator{\Proj}{Proj}
\DeclareMathOperator{\ct}{ct}
\DeclareMathOperator{\ty}{type}

\author[O. Tout]{Omar Tout}
\address{Instytut Matematyczny, Polska Akademia Nauk,
ul. Śniadeckich 8,
00-656 Warszawa,
Poland}
\email{otout@impan.pl}

\title[The center of the wreath product of symmetric groups algebra]{The center of the wreath product of\\ symmetric groups algebra}

\thanks{This research is supported by Narodowe Centrum Nauki, grant number 2017/26/A/ST1/00189.}
 
\begin{document}
\maketitle
\begin{abstract} We consider the wreath product of two symmetric groups as a group of blocks permutations and we study its conjugacy classes. We give a polynomiality property for the structure coefficients of the center of the wreath product of symmetric groups algebra. This allows us to recover an old result of Farahat and Higman about the polynomiality of the structure coefficients of the center of the symmetric group algebra and to generalize our recent result about the polynomiality property of the structure coefficients of the center of the hyperoctahedral group algebra. A particular attention is paid to the cases when the blocks contain two or three elements.
\end{abstract}

\section{Introduction}

The conjugacy classes of the symmetric group $\mathcal{S}_n$ can be indexed by partitions of $n.$ The conjugacy class associated to a partition $\lambda$ is the set of all permutations with cycle-type $\lambda.$ The center of the symmetric group algebra is the algebra over $\mathbb{C}$ generated by the conjugacy classes of the symmetric group. Its structure coefficients have nice combinatorial properties. In \cite{FaharatHigman1959}, Farahat and Higman, gave a polynomiality property for the structure coefficients of the center of the symmetric group algebra. By introducing partial permutations in \cite{Ivanov1999}, Ivanov and Kerov, gave a combinatorial proof to this result.\\

We introduce in this paper the group $\mathcal{B}_{kn}^k$ which permutes $n$ blocks of $k$ elements each. The permutation of the $k$ elements in each block is allowed. This group is the symmetric group $\mathcal{S}_n$ if $k=1$ and in case $k=2$ it is the hyperoctahedral group $\mathcal{H}_n$ on $2n$ elements. In general, we show that the group $\mathcal{B}_{kn}^k$ is isomorphic to the wreath product $\mathcal{S}_k\sim \mathcal{S}_n$ of the symmetric group $\mathcal{S}_k$ by the symmetric group $\mathcal{S}_n.$ It is well known that the conjugacy classes of $\mathcal{H}_n$ are indexed by pairs of partitions $(\lambda,\delta)$ verifying $|\lambda|+|\delta|=n,$ see \cite{geissinger1978representations} and \cite{stembridge1992projective}. We show that in general, for any fixed integer $k,$ the conjugacy classes of the group $\mathcal{B}_{kn}^k$ are indexed by families of partitions $\lambda=(\lambda(\rho))_{\rho\vdash k}$ indexed by the set of partitions of $k$ such that the sum of the sizes of all $\lambda(\rho)$ equals $n.$ This comes with no surprise since it was shown by Specht in \cite{specht1932verallgemeinerung} that the conjugay classes of $\mathcal{S}_k \sim \mathcal{S}_n$ are indexed by families of partitions indexed by the set of partitions of $k,$ see \cite{kerber2006representations} and \cite{McDo} for more information about this fact.\\

Recently, in \cite{Tout2017}, we developed a framework in which the polynomiality property for double-class algebras, and subsequently centers of groups algebra, holds. In particular, we showed that our framework contains the sequence of the symmetric groups and that of the hyperoctahedral groups. Thus we obtained again the result of Farahat and Higman for the structure coefficients of the center of the symmetric group algebra. In addition we gave a polynomiality property for the structure coefficients of the center of the hyperoctahedral group algebra in \cite[Section 6.2]{Tout2017}.

In this paper we show that the general framework we gave in \cite{Tout2017} contains the sequence of groups $(\mathcal{B}_{kn}^k)_n$ when $k$ is a fixed integer. Thus, we will be able to give a polynomiality property for the structure coefficients of the center of the group $\mathcal{B}_{kn}^k$ algebra. A particular attention to the cases $k=2$ (hyperoctahedral group) and $k=3$ is given. The question whether the partial permutation concept of Ivanov and Kerov can be generalized to obtain a combinatorial proof to this result arises directly. We will try to answer this question in another paper.\\

The paper is organized as follows. In Section \ref{sec_2}, we review all necessary definitions of partitions and we describe the conjugacy classes of the symmetric group. Then, we define explicitly the group $\mathcal{B}_{kn}^k$ in section  \ref{sec_3} and we study in details its conjugacy classes. Then, we show that it is isomorphic to the wreath product $\mathcal{S}_k\sim \mathcal{S}_n.$ In sections \ref{sec_hyp} and \ref{seck=3} a special treatment is given for the cases $k=2$ and $k=3$ respectively. The last section contains our main result, that is a polynomiality property for the structure coefficients of the center of the group $\mathcal{B}_{kn}^k$ algebra. In addition some examples are given for the cases $k=2$ and $k=3.$

\section{Partitions and conjugacy classes of the Symmetric group}\label{sec_2}

If $n$ is a positive integer, we denote by $\mathcal{S}_n$ the symmetric group of permutations on the set $[n]:=\lbrace 1,2,\cdots,n\rbrace.$ A \textit{partition} $\lambda$ is a list of integers $(\lambda_1,\ldots,\lambda_l)$ where $\lambda_1\geq \lambda_2\geq\ldots \lambda_l\geq 1.$ The $\lambda_i$ are called the \textit{parts} of $\lambda$; the \textit{size} of $\lambda$, denoted by $|\lambda|$, is the sum of all of its parts. If $|\lambda|=n$, we say that $\lambda$ is a partition of $n$ and we write $\lambda\vdash n$. The number of parts of $\lambda$ is denoted by $l(\lambda)$. We will also use the exponential notation $\lambda=(1^{m_1(\lambda)},2^{m_2(\lambda)},3^{m_3(\lambda)},\ldots),$ where $m_i(\lambda)$ is the number of parts equal to $i$ in the partition $\lambda.$ In case there is no confusion, we will omit $\lambda$ from $m_i(\lambda)$ to simplify our notation. If $\lambda=(1^{m_1(\lambda)},2^{m_2(\lambda)},3^{m_3(\lambda)},\ldots,n^{m_n(\lambda)})$ is a partition of $n$ then $\sum_{i=1}^n im_i(\lambda)=n.$ We will dismiss $i^{m_i(\lambda)}$ from $\lambda$ when $m_i(\lambda)=0.$ For example, we will write $\lambda=(1^2,3,6^2)$ instead of $\lambda=(1^2,2^0,3,4^0,5^0,6^2,7^0).$ If $\lambda$ and $\delta$ are two partitions we define the \textit{union} $\lambda \cup \delta$ and subtraction $\lambda \setminus \delta$ (if exists) as the following partitions:
$$\lambda \cup \delta=(1^{m_1(\lambda)+m_1(\delta)},2^{m_2(\lambda)+m_2(\delta)},3^{m_3(\lambda)+m_3(\delta)},\ldots).$$
$$\lambda \setminus \delta=(1^{m_1(\lambda)-m_1(\delta)},2^{m_2(\lambda)-m_2(\delta)},3^{m_3(\lambda)-m_3(\delta)},\ldots) \text{ if $m_i(\lambda)\geq m_i(\delta)$ for any $i.$ }$$
A partition is called \textit{proper} if it does not have any part equal to 1. The proper partition associated to a partition $\lambda$ is the partition $\bar{\lambda}:=\lambda \setminus (1^{m_1(\lambda)})=(2^{m_2(\lambda)},3^{m_3(\lambda)},\ldots).$ 

\bigskip

The \textit{cycle-type} of a permutation of $\mathcal{S}_n$ is the partition of $n$ obtained from the lengths of the cycles that appear in its decomposition into product of disjoint cycles. For example, the permutation $(2,4,1,6)(3,8,10,12)(5)(7,9,11)$ of $\mathcal{S}_{12}$ has cycle-type $(1,3,4^2).$ In this paper we will denote the cycle-type of a permutation $\omega$ by $\ct(\omega).$ It is well known that two permutations of $\mathcal{S}_n$ belong to the same conjugacy class if and only if they have the same cycle-type. Thus the conjugacy classes of the symmetric group $\mathcal{S}_n$ can be indexed by partitions of $n.$ If $\lambda=(1^{m_1(\lambda)},2^{m_2(\lambda)},3^{m_3(\lambda)},\ldots,n^{m_n(\lambda)})$ is a partition of $n,$ we will denote by $C_\lambda$ the conjugacy class of $\mathcal{S}_n$ associated to $\lambda:$
$$C_\lambda:=\lbrace \sigma\in \mathcal{S}_n \text{ $\mid$ } \ct(\sigma)=\lambda \rbrace.$$
The cardinal of $C_\lambda$ is given by:
$$|C_\lambda|=\frac{n!}{z_\lambda},$$
where
$$z_\lambda:=1^{m_1(\lambda)}m_1(\lambda)!2^{m_2(\lambda)}m_2(\lambda)!\cdots n^{m_n(\lambda)}m_n(\lambda)!.$$

\section{The conjugacy classes of the group $\mathcal{B}_{kn}^{k}$}\label{sec_3}

In this section we define the group $\mathcal{B}_{kn}^{k}$ then we study its conjugacy classes. We show in Proposition \ref{class_conj_classes} that these latter are indexed by families of partitions indexed by all partitions of $k.$ Then in Proposition \ref{Prop_size_conj} we give an explicit formula for the size of any of its conjugacy classes. The group $\mathcal{B}_{kn}^{k}$ is isomorphic to the wreath product $\mathcal{S}_k\sim \mathcal{S}_n$ as will be shown in Proposition \ref{isom}. However, we decided to work with this copy of the wreath product in this paper since it seems very natural to present our main result.\\

If $i$ and $k$ are two positive integers, we denote by $p_{k}(i)$ the following set of size $k:$ 
$$p_k(i):=\lbrace (i-1)k+1, (i-1)k+2, \cdots , ik\rbrace.$$ 

The above set $p_k(i)$ will be called a $k$-tuple in this paper. We define the group $\mathcal{B}_{kn}^{k}$ to be the subgroup of $\mathcal{S}_{kn}$ formed by permutations that send each set of the form $p_{k}(i)$ to another set with the same form:
$$\mathcal{B}_{kn}^{k}:=\lbrace w \in \mathcal{S}_{kn}; \ \forall \ 1 \leq r \leq n, \ \exists \ 1 \leq r' \leq n \text{ such that } w(p_{k}(r))=p_{k}(r')\rbrace.$$

\begin{ex} $\begin{pmatrix}
	1&2&3&&4&5&6\\
	1&3&2&&6&5&4
	\end{pmatrix}\in \mathcal{B}_{6}^{3}$ but $\begin{pmatrix}
	1&2&3&&4&5&6\\
	1&3&6&&2&4&6
	\end{pmatrix}\notin \mathcal{B}_{6}^{3}.$
\end{ex}

The group $\mathcal{B}_{kn}^{k}$ as it is defined here appears in \cite[Section 7.4]{Tout2017}. When $k=1,$ it is clear that $\mathcal{B}_{n}^1$ is the symmetric group $\mathcal{S}_n.$ When $k=2,$ the group $\mathcal{B}_{2n}^2$ is the hyperoctahedral group $\mathcal{H}_n$ on $2n$ elements, see \cite{toutejc} where the author treats $\mathcal{H}_n$ as being the group $\mathcal{B}_{2n}^2.$ It would be clear to see that the order of the group $\mathcal{B}_{kn}^{k}$ is 
$$|\mathcal{B}_{kn}^{k}|=(k!)^nn!$$ 

The decomposition of a permutation $\omega\in \mathcal{B}_{kn}^{k}$ into product of disjoint cycles has remarkable patterns. To see this, fix a $k$-tuple $p_{k}(i)$ and a partition $\rho=(\rho_1,\rho_2,\cdots,\rho_l)$ of $k.$ Suppose now that while writing $\omega$ as product of disjoint cycles we get among the cycles the following pattern:
\begin{equation}\label{canonical_decom}
\mathcal{C}_1=(a_1,\cdots,a_2,\cdots,a_{\rho_1},\cdots),
\end{equation}
$$
\mathcal{C}_2=(a_{\rho_1+1},\cdots,a_{\rho_1+2},\cdots,a_{\rho_1+\rho_2},\cdots),
$$
$$\vdots$$
$$
\mathcal{C}_l=(a_{\rho_1+\cdots+\rho_{l-1}+1},\cdots,a_{\rho_1+\cdots+\rho_{l-1}+2},\cdots,a_{\rho_1+\cdots+\rho_{l-1}+\rho_l},\cdots),
$$
where $$\lbrace a_1,a_2,\cdots,a_{\rho_1+1},\cdots,a_{\rho_1+\cdots+\rho_{l-1}+\rho_l}\rbrace=p_k(i) \text{ for a certain $i\in [n]$}.$$
We should remark that since $\omega\in \mathcal{B}_{kn}^{k},$ if we consider $b_j=\omega(a_j)$ for any $1\leq j\leq |\rho|$ then there exists $r\in [n]$ such that:
$$p_{k}(r)=\lbrace b_1,b_2,\cdots,b_{\rho_1+1},\cdots,b_{\rho_1+\cdots+\rho_{l-1}+\rho_l}\rbrace.$$
This can be redone till we reach $a_2$ which implies that in cycle $\mathcal{C}_1,$ we have the same number of elements between $a_i$ and $a_{i+1}$ for any $1\leq i\leq \rho_1-1.$ Thus the size of the cycle $\mathcal{C}_1$ is a multiple of $\rho_1,$ say $|\mathcal{C}_1|=m\rho_1.$ The same can be done for all the other cycles $\mathcal{C}_i$ and in fact for any $1\leq i\leq l,$  $|\mathcal{C}_i|=m\rho_i.$  In addition, if we take the set of all the elements that figure in the cycles $\mathcal{C}_i$ we will get a disjoint union of $m$ $k$-tuples. That means:
$$\sum_{j=1}^{l}|\mathcal{C}_j|=\sum_{j=1}^{l} m\rho_j=mk.$$
This integer $m$ can be obtained from any of the cycles $\mathcal{C}_j$ by looking to the number of elements that separate two consecutive elements of the same $k$-tuple. Now construct the partition $\omega(\rho)$ by grouping all the integers $m$ as above.

\begin{remark} We should pay the reader's attention to two important remarks after this construction. First of all, for any partition $\rho$ of $k$ there are $rk$ elements involved when adding the part $r$ to $\omega(\rho).$ That means: 
$$\sum_{\rho\vdash k}\sum_{r\geq 1}krm_r(\omega(\rho))=kn \text{ which implies } \sum_{\rho\vdash k}\sum_{r\geq 1}rm_r(\omega(\rho))=n. $$
Now let $p_\omega$ denotes the blocks permutation of $n$ associated to $\omega.$ That is $p_\omega(i)=j$ whenever $\omega(p_k(i))=p_k(j).$ Adding a part $r$ to one of the partitions $\omega(\rho)$ implies that $r$ is the length of one of the cycles of $p_\omega.$ Thus we have:
\begin{equation*}
\bigcup_{\rho\vdash k}\omega(\rho)=\ct(p_\omega).
\end{equation*}
In other words, one may first find the cycle-type of $p_\omega$ and then distribute all its parts between the partitions $\omega(\rho)$ according to the above construction. 
\end{remark}

\begin{ex}\label{k=3,equivex}
Consider the following permutation, written in one-line notation, $\omega$ of $\mathcal{B}^3_{24}$ 

$$\omega=12~~10~~11\,\,\,\, 20~~21~~19\,\,\,\, 8~~7~~9\,\,\,\, 1~~2~~3\,\,\,\, 16~~18~~17 \,\,\,\,15~~14~~13 \,\,\,\, 5~~4~~6 \,\,\,\, 22~~23~~24$$
Its decomposition into product of disjoint cycles is:
$$\omega=\color{red}(1,12,3,11,2,10)\color{blue}(4,20)(5,21,6,19)\color{green}(7,8)(9)\color{brown}(13,16,15,17,14,18)\color{black}(22)(23)(24).$$
The first red cycle contains all the elements of $p_3(1)$ thus it contributes to $\omega(3).$ In it, there are two elements between $1$ and $3$ thus we should add $2$ to the partition $\omega(3).$  The brown cycle contributes to $\omega(3)$ by $2$ also, thus $\omega(3)=(2,2).$ By looking to the blue cycles we see that $5$ and $6$ belong to the same cycle while $4$ belongs to the other, thus these cycles will contribute to $\omega(2,1).$ The distance between $5$ and $6$ is two (the same can be done by looking to the cycle containing $4$ but in this case since $4$ is the only element of $p_3(3)$ in this cycle we count the distance between $4$ and $4$ which is $2$) which means that $\omega(2,1)$ contains a cycle of length $2.$ The green cycle will add $1$ to $\omega(2,1)$ to become the partition $(2,1).$ The black cycles give $\omega(1,1,1)=(1).$ The reader should remark that $p_\omega$ is the permutation $(1,4)(2,7)(3)(5,6)(8)$ of $8$ and that:
$$\omega(3)\cup \omega(2,1)\cup \omega(1^3)=\ct(p_\omega)=(1^2,2^3).$$
\end{ex}
 
\begin{definition}
If $\omega\in \mathcal{B}_{kn}^k,$ define $\ty(\omega)$ to be the following family of partitions indexed by partitions of $k$
$$\ty(\omega):=(\omega(\rho))_{\rho\vdash k}.$$
\end{definition}

\begin{prop}\label{class_conj_classes}
Two permutations $\alpha$ and $\beta$ of $\mathcal{B}_{kn}^k$ are in the same conjugacy class if and only if they both have the same type.
\end{prop}
\begin{proof}
Suppose $\alpha=\gamma\beta\gamma^{-1}$ for some $\gamma\in \mathcal{B}_{kn}^k$ and fix a partition $\rho$ of $k.$ Suppose that the cycles $\mathcal{C}_1,\cdots,\mathcal{C}_l$ of Equation (\ref{canonical_decom}) contribute to $\beta(\rho)$ then the cycles
\begin{equation*}
\mathcal{C}^{'}_1=(\gamma(a_1),\cdots,\gamma(a_2),\cdots,\gamma(a_{\rho_1}),\cdots),
\end{equation*}
$$
\mathcal{C}^{'}_2=(\gamma(a_{\rho_1+1}),\cdots,\gamma(a_{\rho_1+2}),\cdots,\gamma(a_{\rho_1+\rho_2}),\cdots),
$$
$$\vdots$$
$$
\mathcal{C}^{'}_l=(\gamma(a_{\rho_1+\cdots+\rho_{l-1}+1}),\cdots,\gamma(a_{\rho_1+\cdots+\rho_{l-1}+2}),\cdots,\gamma(a_{\rho_1+\cdots+\rho_{l-1}+\rho_l}),\cdots),
$$
will contribute to $\alpha(\rho)$ and they have respectively the same lengths as $\mathcal{C}_1,\cdots,\mathcal{C}_2.$ This proves the first implication.

Conversely, $\ty(\alpha)=\ty(\beta)$ means that $\alpha(\rho)=\beta(\rho)$ for any partition $\rho$ of $k.$  In order to simplify, we will look at the elements of two fixed $k$-tuples, say $p_k(1)=\lbrace 1,2,\cdots, k\rbrace$ for $\alpha$ and $p_k(2)=\lbrace k+1,k+2,\cdots ,2k\rbrace$ for $\beta,$ such that the distribution of the elements of $p_k(1)$ among the cycle decomposition in $\alpha$ is similar to that of the elements of $p_k(2)$ in $\beta.$ In other words, the cycle decompositions of $\alpha$ and $\beta$ are as follows:
$$
\alpha=(1,\cdots,2,\cdots,\rho_1,\cdots)(\rho_1+1,\cdots,\rho_1+2,\cdots,\rho_1+\rho_2,\cdots)\cdots$$
$$\cdots(\rho_1+\cdots+\rho_{l-1}+1,\cdots,\rho_1+\cdots+\rho_{l-1}+2,\cdots,\rho_1+\cdots+\rho_{l-1}+\rho_l,\cdots)$$
and
$$
\beta=(k+1,\cdots,k+2,\cdots,k+\rho_1,\cdots)(k+\rho_1+1,\cdots,k+\rho_1+2,\cdots,k+\rho_1+\rho_2,\cdots)\cdots$$
$$\cdots(k+\rho_1+\cdots+\rho_{l-1}+1,\cdots,k+\rho_1+\cdots+\rho_{l-1}+2,\cdots,k+\rho_1+\cdots+\rho_{l-1}+\rho_l,\cdots)
$$
Construct $\gamma$ to be the permutation that orderly takes each element of the above cycles of $\alpha$ to each element with the same order in $\beta,$ that is:
$$\gamma(1)=k+1, \gamma(\alpha(1))=\beta(k+1),\cdots \gamma(\rho_1)=k+\rho_1,\cdots$$
It is then clear that $\alpha=\gamma^{-1}\beta\gamma$ and $\gamma\in \mathcal{B}_{kn}^k.$
\end{proof}

\begin{prop}\label{Prop_size_conj}
Let $\omega$ be a permutation of $\mathcal{B}_{kn}^k$ then the size of the conjugacy class of $\omega$ is given by:
$$\frac{n!(k!)^n}{\displaystyle \prod_{\rho\vdash k}z_{\omega(\rho)}z_\rho^{l(\omega(\rho))}}.$$
\end{prop}

\begin{proof} Suppose that $\ct(p_\omega)=\lambda,$ where $\lambda$ is a partition of $n.$ We will explain how to obtain all the elements $\gamma\in \mathcal{B}_{kn}^k$ that have the same type as $\omega.$ First of all, $p_\gamma$ must be decomposed in cycles with sizes equivalent to those of $p_\omega.$ In other words, $p_\gamma$ should have the same cycle-type as $p_\omega$ and there are 
$$\frac{n!}{z_{\lambda}}$$
choices to make $p_\gamma.$ 

Now suppose that $m_r(\lambda)\neq 0$ for some $r\in [n].$ That means that there are exactly $m_r(\lambda)$ cycles of length $r$ distributed in the family $(\omega(\rho))_{\rho\vdash k}.$ Suppose they all appear in the partitions $\omega(\rho_1),\cdots, \omega(\rho_m).$ In order to have $\ty(\gamma)=\ty(\omega),$ we should make $\gamma(\rho_1)=\omega(\rho_1),$ $\gamma(\rho_2)=\omega(\rho_2)$ and so on. That is $\gamma(\rho_1)$ should have parts equal to $r$ as many as $\omega(\rho_1)$ has and so on. The total number of $\gamma$ verifying these conditions is
$$\frac{m_r(\lambda)!}{\displaystyle \prod_{1\leqslant j\leqslant m}m_r(\omega(\rho_j))!}=\frac{m_r(\lambda)!}{\displaystyle \prod_{\rho\vdash k}{m_r(\omega(\rho))!}}$$

It remains to see how many $\gamma\in \mathcal{B}_{kn}^k$ one can make when he nows all $\gamma(\rho).$ For this fix a partition $\rho_1$ of $k$ such that $\gamma(\rho_1)\neq \emptyset$ and suppose $m_1$ is a part of $\gamma(\rho_1).$ There will be $m_1$ $k$-tuples involved in the construction of $\gamma$ now. Choose one of them then distribute its elements according to $\rho_1$ in $\frac{k!}{z_{\rho_1}}$ ways. To complete $\gamma,$ there will be $(k!)^{m_1-1}$ choices for the elements between any two consecutive elements of the fixed $k$-tuples in any chosen cycle.

By bringing together all the above arguments, the size of the conjugacy class of $\omega$ is:

$$\frac{n!}{z_{\lambda}}\displaystyle \prod_{r,m_r(\lambda)\neq 0} \frac{m_r(\lambda)!}{\displaystyle \prod_{\rho\vdash k}{m_r(\omega(\rho))!}}\displaystyle \frac{(k!)^{n}}{z_{\rho}^{l(\omega(\rho))}}=\frac{n!(k!)^n}{\displaystyle \prod_{\rho\vdash k}z_{\omega(\rho)}z_\rho^{l(\omega(\rho))}}.$$
This equality is due to the fact that 
$$l(\lambda)=\sum_{r\geq 1}m_r(\lambda)=\sum_{r\geq 1}\sum_{\rho\vdash k}m_r(\omega(\rho)).$$
\end{proof}

We turn now to show that the group $\mathcal{B}_{kn}^k$ is isomorphic to the wreath product $\mathcal{S}_k \sim \mathcal{S}_n.$ The wreath product $\mathcal{S}_k \sim \mathcal{S}_n$ is the group with underlying set $\mathcal{S}_k^n\times \mathcal{S}_n$ and product defined as follows:
$$((\sigma_1,\sigma_2,\cdots ,\sigma_n); p).((\epsilon_1,\epsilon_2,\cdots ,\epsilon_n); q)=((\sigma_1\epsilon_{p^{-1}(1)},\sigma_2\epsilon_{p^{-1}(2)},\cdots ,\sigma_n\epsilon_{p^{-1}(n)});pq),$$
for any $((\sigma_1,\sigma_2,\cdots ,\sigma_n); p),((\epsilon_1,\epsilon_2,\cdots ,\epsilon_n); q)\in \mathcal{S}_k^n\times \mathcal{S}_n.$ The identity in this group is $(1;1):=((1_k,1_k,\cdots ,1_k); 1_n),$ where $1_i$ denotes the identity function of $\mathcal{S}_i.$ The inverse of an element $((\sigma_1,\sigma_2,\cdots ,\sigma_n); p)\in \mathcal{S}_k \sim \mathcal{S}_n$ is given by $$((\sigma_1,\sigma_2,\cdots ,\sigma_n); p)^{-1}:=((\sigma^{-1}_{p(1)},\sigma^{-1}_{p(2)},\cdots ,\sigma^{-1}_{p(n)}); p^{-1}).$$

For each permutation $\omega\in \mathcal{B}^k_{kn}$ and each integer $i\in [n],$ define $\omega_i$ to be the normalized restriction of $\omega$ on the block $p_\omega^{-1}(i).$ That is:

$$\begin{array}{ccccc}
\omega_i & : &[k] & \to & [k] \\
& & b & \mapsto &  \omega_i(b):=\omega\big(k(p_\omega^{-1}(i)-1)+b\big)\%k,\\
\end{array}$$
where $\%$ means that the integer is taken modulo $k$ (for a multiple of $k$ we use $k$ instead of $0$).

\begin{ex}\label{ex_psi} 
Consider the following permutation $\alpha$ of $\mathcal{B}^3_{18}:$ 
$$
\alpha=\begin{pmatrix}
1&2&3&4&5&6&7&8&9&10&11&12&13&14&15&16&17&18\\
12&10&11&5&6&4&8&7&9&15&13&14&16&18&17&3&2&1
\end{pmatrix}.
$$
The blocks permutation associated to $\alpha$ is $p_\alpha=(1,4,5,6)(2)(3).$ In addition, we have:
$$\alpha_1=(1,3),~~\alpha_2=(1,2,3),~~\alpha_3=(1,2),~~\alpha_4=(1,3,2),~~\alpha_5=(1,3,2) \text{ and }\alpha_6=(2,3).$$
\end{ex}

\begin{prop}\label{isom}
The application 
$$\begin{array}{ccccc}
\psi & : &\mathcal{B}^k_{kn} & \to & \mathcal{S}_k\sim \mathcal{S}_n \\
& & \omega & \mapsto &  \psi(\omega):=((\omega_1,\cdots ,\omega_n);p_\omega),\\
\end{array}$$
is a group isomorphism.
\end{prop}
\begin{proof}
$\psi$ is clearly a bijection with inverse given by:

$$\begin{array}{ccccc}
\phi & : &\mathcal{S}_k\sim \mathcal{S}_n & \to & \mathcal{B}^k_{kn} \\
& & ((\sigma_1,\sigma_2,\cdots ,\sigma_n); p) & \mapsto &  \sigma,\\
\end{array}$$
where $\sigma\big(k(a-1)+b\big)=k(p(a)-1)+\sigma_{p(a)}(b)$ for any $a\in [n]$ and any $b\in [k].$ It remains to show that if $x=((\sigma_1,\sigma_2,\cdots ,\sigma_n); p)$ and $y=((\epsilon_1,\epsilon_2,\cdots ,\epsilon_n); q)$ are two elements of $\mathcal{S}_k\sim \mathcal{S}_n$ then $\phi(x.y)=\phi(x)\phi(y).$ To prove this let $a\in [n]$ and $b\in [k].$ On the right hand we have:
$$\phi(x)\phi(y)\big(k(a-1)+b\big)=\phi(x)\big(k(q(a)-1)+\epsilon_{q(a)}(b)\big)=k(p(q(a))-1)+\sigma_{p(q(a))}(\epsilon_{q(a)}(b))$$
and on the left hand we have:
$$\phi(xy)\big(k(a-1)+b\big)=k((pq)(a)-1)+\sigma_{(pq)(a)}\epsilon_{p^{-1}((pq)(a))}(b).$$ This shows that both functions $\phi(x.y)$ and $\phi(x)\phi(y)$ are equal which finishes the proof.
\end{proof}

\begin{ex} Recall the permutation $\alpha\in \mathcal{B}^3_{18}$ of Example \ref{ex_psi} and consider the following permutation $\beta$ of the same group: 
$$
\beta=\begin{pmatrix}
1&2&3&4&5&6&7&8&9&10&11&12&13&14&15&16&17&18\\
4&5&6&18&17&16&8&9&7&1&2&3&12&11&10&15&14&13
\end{pmatrix}.
$$
We have:
$$ 
\alpha\beta=\begin{pmatrix}
1&2&3&4&5&6&7&8&9&10&11&12&13&14&15&16&17&18\\
5&6&4&1&2&3&7&9&8&12&10&11&14&13&15&17&18&16
\end{pmatrix},
$$
$$\psi(\alpha)=\Big(\big( (1,3),(1,2,3),(1,2),(1,3,2),(1,3,2),(2,3)\big);(1,4,5,6)(2)(3)\Big),$$
$$\psi(\beta)=\Big(\big( 1,1,(1,2,3),(1,3),(1,3),(1,3)\big);(1,2,6,5,4)(3)\Big)$$
and 
$$\psi(\alpha\beta)=\Big(\big( 1,(1,2,3),(2,3),(1,3,2),(1,2),(1,2,3)\big);(1,2)(3)(4)(5)(6)\Big).$$
Now we can easily verify that $\psi(\alpha\beta)=\psi(\alpha).\psi(\beta).$

\end{ex}

\section{Special cases} 
In this section we will treat two special cases, the case of $k=2$ and that of $k=3.$ We will see that when $k=2,$ the group $\mathcal{B}_{2n}^2$ is the hyperoctahedral group on $2n$ elements. There are many papers in the literature, see \cite{stembridge1992projective} and \cite{geissinger1978representations} for examples, that study the representations of the hyperoctahedral group. Especially we are interested here in studying its conjugacy classes. We will recover some of its nice properties using our approach.

\subsection{Case $k=2:$ The hyperoctahedral group}\label{sec_hyp}

When $k=2$ there are only two types of partitions of $2,$ mainly $\lambda_1=(1^2)$ and $\lambda_2=(2).$ Now we are going to see how things go in the context of $\mathcal{B}^2_{2n}.$ We use the same arguments given in \cite[Section 6.2]{Tout2017}. If $a\in p_2(i)$, we shall denote by $\overline{a}$ the element of the set $p_2(i)\setminus \lbrace a\rbrace.$ Therefore, we have, $\overline{\overline{a}}=a$ for any $a\in [2n].$ As seen in our general construction, the cycle decomposition of a permutation of $\mathcal{B}^2_{2n}$ will have two types of cycles. To see this, suppose that $\omega$ is a permutation of $\mathcal{B}^2_{2n}$ and take the following cycle $\mathcal{C}$ of its decomposition:
$$\mathcal{C}=(a_1,\cdots ,a_{l}).$$
We distinguish two cases :
\begin{enumerate}
\item $\overline{a_1}$ appears in the cycle $\mathcal{C},$ for example $a_j=\overline{a_1}.$ Since $\omega\in\mathcal{B}^2_{2n}$ and $\omega(a_1)=a_2,$ we have $\omega(\overline{a_1})=\overline{a_2}=\omega(a_j).$ Likewise, since $\omega(a_{j-1})=\overline{a_1},$ we have $\omega(\overline{a_{j-1}})=a_1$ which means that $a_{l}=\overline{a_{j-1}}.$ Therefore,
$$\mathcal{C}=(a_1,\cdots a_{j-1},\overline{a_1},\cdots,\overline{a_{j-1}})$$
and $l=2(j-1)$ is even. We will denote such a cycle by $(\mathcal{O},\overline{\mathcal{O}}).$
\item $\overline{a_1}$ does not appear in the cycle $\mathcal{C}.$ Take the cycle $\overline{\mathcal{C}}$ which contains $\overline{a_1}.$ Since $\omega(a_1)=a_2$ and $\omega\in \mathcal{B}^2_{2n}$, we have $\omega(\overline{a_1})=\overline{a_2}$ and so on. That means that the cycle $\overline{\mathcal{C}}$ has the following form,
$$\overline{\mathcal{C}}=(\overline{a_1},\overline{a_2},\cdots ,\overline{a_{l}})$$
and that $\mathcal{C}$ and $\overline{\mathcal{C}}$ appear in the cycle decomposition of $\omega.$
\end{enumerate} 

The cycles of the first case will contribute to $\omega(\lambda_2)$ while those of the second will contribute to $\omega(\lambda_1).$ Suppose now that the cycle decomposition of a permutation $\omega$ of $\mathcal{B}^2_{2n}$ is as follows:
$$\omega=\mathcal{C}_1\overline{\mathcal{C}_1}\mathcal{C}_2\overline{\mathcal{C}_2}\cdots \mathcal{C}_k\overline{\mathcal{C}_k}(\mathcal{O}^1,\overline{\mathcal{O}^1})(\mathcal{O}^2,\overline{\mathcal{O}^2})\cdots (\mathcal{O}^l,\overline{\mathcal{O}^l}),$$
where the cycles $\mathcal{C}_i$ (resp. $(\mathcal{O}^j,\overline{\mathcal{O}^j})$) are written decreasingly according to their sizes.
From this decomposition, we obtain that the parts of the partition $\omega(\lambda_1)$ are the sizes of the sets $\mathcal{C}_j,$ while the parts of the partition $\omega(\lambda_2)$ are the sizes of the sets $\mathcal{O}^i:$
$$\omega(\lambda_1)=(|\mathcal{C}_1|,\cdots,|\mathcal{C}_k|),~~\omega(\lambda_2)=(|\mathcal{O}^1|,\cdots,|\mathcal{O}^l|) \text{ and }|\omega(\lambda_1)|+|\omega(\lambda_2)|=n.$$

\begin{ex} Consider the following permutation 
$$\omega=\begin{pmatrix}
1&2&&3&4&&5&6&&7&8&&9&10&&11&12&&13&14&&15&16\\
14&13&&1&2&&16&15&&7&8&&12&11&&10&9&&4&3&&5&6
\end{pmatrix}\in \mathcal{B}^2_{16}.$$ Its decomposition into product of disjoint cycles is as follows:
$$\omega=(1,14,3)(2,13,4)(7)(8)(9,12)(10,11)(5,16,6,15).$$
Then $\omega(\lambda_1)=(3,2,1)$ and $\omega(\lambda_2)=(2).$ 
\end{ex}

Apply Proposition \ref{Prop_size_conj} to obtain the following result.

\begin{cor}\label{size_conj_k=2}
The size of the conjugacy class of a permutation $\omega\in \mathcal{B}^2_{2n}$ is:
$$\frac{2^nn!}{2^{l(\omega(\lambda_1))+l(\omega(\lambda_2))}z_{\omega(\lambda_1)}z_{\omega(\lambda_2)}}.$$
\end{cor}

The above Corollary \ref{size_conj_k=2} is a well known formula for the sizes of the conjugacy classes of the hyperoctahedral group, see \cite{stembridge1992projective} and \cite{geissinger1978representations}.

\subsection{Case $k=3$}\label{seck=3} In a way similar to that of case $k=2,$ the fact that there are only three partitions of $3$ suggests that there are three types of cycles in the decomposition into product of disjoint cycles of a permutation $\omega\in \mathcal{B}^3_{3n}.$

Let $\mathcal{C}$ be a cycle of $\omega\in \mathcal{B}^3_{3n},$ we distinguish the following three cases:
\begin{enumerate}
\item first case: all three elements of a certain $p_3(s)$ belong to $\mathcal{C}.$ For simplicity, suppose:
$$\mathcal{C}=(a_1=1,a_2,a_3,\cdots, a_j=2,a_{j+1},\cdots, a_l=3,a_{l+1},\cdots a_k).$$
Since $\omega\in \mathcal{B}^3_{3n},$ the sets $\lbrace a_2,a_{j+1},a_{l+1}\rbrace,$ $\lbrace a_3,a_{j+2},a_{l+2}\rbrace,\cdots,$ $\lbrace a_{j-1},a_{l-1},a_k\rbrace$ all have the form $p_3(m)$ and thus $\mathcal{C}$ is a cycle of length $3(j-1)$ that contains a union of sets of the form $p_3(r).$
\item second case:  two and only two elements of a certain $p_3(s)$ belong to $\mathcal{C}.$ Say,
$$\mathcal{C}=(a_1=1,a_2,a_3,\cdots, a_j=2,a_{j+1},\cdots , a_k).$$
Since $\omega\in \mathcal{B}^3_{3n},$ there exists integers $b_i,$ $1\leq i\leq j-1,$ such that 
$$\lbrace a_1,a_{j}\rbrace=p_3(b_1)\setminus \lbrace c_1\rbrace,$$
$$\lbrace a_2,a_{j+1}\rbrace=p_3(b_2)\setminus \lbrace c_2\rbrace,$$
$$\vdots $$
$$\lbrace a_{j-1},a_k\rbrace=p_3(b_{j-1})\setminus \lbrace c_{j-1}\rbrace,$$
and another cycle $(c_1,c_2,\cdots ,c_{j-1})$ should thus appear in the decomposition of $\omega.$
\item third case:  all three elements of a certain $p_3(s)$ belong to different cycles. In this case, all three cycles will have the same lengths and each one of them will contain elements belonging to different triplets.
\end{enumerate} 

Now for any permutation $\omega\in \mathcal{B}^3_{3n},$ define $\gamma_\omega$ to be the partition obtained from the lengths divided by three of the cycles of the first case, $\beta_\omega$ the partition obtained from the lengths of the cycles of the second case divided by two and $\alpha_\omega$ the partition obtained from the lengths of the cycles of the third case. It would be clear that 
$$\omega(\lambda_1)=\alpha_\omega,~~ \omega(\lambda_2)=\beta_\omega,~~ \omega(\lambda_3)=\gamma_\omega \text{ and } |\gamma_\omega|+|\beta_\omega|+|\alpha_\omega|=n.$$ 
For example, for the permutation $\omega$ of Example \ref{k=3,equivex}, we have:
$$\gamma_\omega=(2,2),~~\beta_\omega=(2,1) \text{ and } \alpha_\omega=(1).$$

Using Proposition \ref{Prop_size_conj}, we obtain the following result.

\begin{cor}\label{size_con_k=3} The size of the conjugacy class of $\omega\in \mathcal{B}_{3n}^3$ is:

$$\frac{(3!)^nn!}{(3!)^{l(\alpha_\omega)}z_{\alpha_\omega}2^{l(\beta_\omega)}z_{\beta_\omega}3^{l(\gamma_\omega)}z_{\gamma_\omega}}=2^{n-l(\alpha_\omega)-l(\beta_\omega)}3^{n-l(\alpha_\omega)-l(\gamma_\omega)}.\frac{n!}{z_{\alpha_\omega}z_{\beta_\omega}z_{\gamma_\omega}}.$$

\end{cor}

\section{The center of the group $\mathcal{B}_{kn}^k$ algebra}\label{sec_6}

In this section, we present in Theorem \ref{main_the} a polynomiality property for the structure coefficients of the center of the group $\mathcal{B}_{kn}^k$ algebra. This can be seen as a generalisation of the Farahat and Higman result in \cite{FaharatHigman1959} and our result in \cite{Tout2017} that gave polynomiality properties for the structure coefficients of the center of the symmetric group and the hyperoctahedral group algebras respectively. A special treatment for the cases $k=2$ and $k=3$ is given.
\subsection{The algebra $Z(\mathbb{C}[\mathcal{B}_{kn}^k])$} 
The center of the group $\mathcal{B}_{kn}^k$ algebra will be denoted $Z(\mathbb{C}[\mathcal{B}_{kn}^k]).$ It is the algebra over $\mathbb{C}$ spanned by the "formal sum of elements of the" conjugacy classes of $\mathcal{B}_{kn}^k.$ According to Proposition \ref{class_conj_classes}, these are indexed by families of partitions $x=(x(\lambda))_{\lambda\vdash k}$ satisfying the property:
\begin{equation}\label{co}
|x|:=\sum_{\lambda\vdash k}|x(\lambda)|=n
\end{equation}
and for each such a family its associated conjugacy class $C_{x}$ is
$$C_{x}=\lbrace t\in \mathcal{B}_{kn}^k \text{ such that }\ty(t)=x\rbrace$$
while its formal sum of elements is
$$\mathbf{C}_x:=\sum_{t\in C_x}t.$$
From now on and unless stated otherwise, $x$ is a family of partition would mean $x=(x(\lambda))_{\lambda\vdash k}$ with the condition \ref{co}.

Let $x$ and $y$ be two family of partitions. In the algebra $Z(\mathbb{C}[\mathcal{B}_{kn}^k]),$ the product $\mathbf{C}_x\mathbf{C}_y$ can be written as a linear combination as following 
\begin{equation}\label{struc_coef}
\mathbf{C}_x\mathbf{C}_y=\sum_{z}c_{xy}^z \mathbf{C}_z
\end{equation}
where $z$ runs through all the families of partitions. The coefficients $c_{xy}^z$ that appear in this equation are called the structure coefficients of the center of the group $\mathcal{B}_{kn}^k$ algebra. 

\subsection{Polynomiality of the structure coefficients of $Z(\mathbb{C}[\mathcal{B}_{kn}^k])$} When $k=1,$ Farahat and Higman were the first to give a polynomiality property for the structure coefficients of the center of the symmetric group algebra in \cite{FaharatHigman1959}. In the case of the center of the hyperoctahedral group algebra, we gave a polynomiality property for its structure coefficients in \cite{Tout2017}. The goal of this section is to generalize these results and show that the structure coefficients of $Z(\mathbb{C}[\mathcal{B}_{kn}^k])$ have a polynomiality property for any fixed $k.$ \\

The most natural way to see a permutation $\omega\in \mathcal{B}_{kn}^k$ as an element of $\omega\in \mathcal{B}_{k(n+1)}^k$ is by extending it by identity. By doing so, the new permutation will have the same type as $\omega$ except that the partition $\omega(1^k)$ will become $\omega(1^k)\cup (1).$

\begin{definition}\label{def_propre}
A family of partitions $x=(x(\lambda))_{\lambda\vdash k}$ is said to be proper if and only if the partition $x(1^k)$ is proper. If $x=(x(\lambda))_{\lambda\vdash k}$ is a proper family of partitions such that $|x|<n,$
we define $C_{x}(n)$ to be the set of elements $t\in \mathcal{B}_{kn}^k$ that have type equals to $x$ except that $x(1^k)$ is replaced by $x(1^k)\cup (1^{n-|x|}).$
\end{definition}

If $x=(x(\lambda))_{\lambda\vdash k}$ is a proper family of partitions such that $|x|=n_0$ then for any $n>n_0$ we have by Proposition \ref{Prop_size_conj} the following result:

\begin{eqnarray*}
|C_x(n)|&=&\frac{n!(k!)^n}{z_{x(1^k)\cup (1^{n-n_0})}(k!)^{l(x(1^k))+n-n_0}\displaystyle \prod_{\lambda\vdash k, \lambda\neq (1^k)}z_{x(\lambda)}z_\lambda^{l(x(\lambda))}}\\
&=&\frac{n!(k!)^{n_0-l(x(1^k))}}{z_{x(1^k)}(n-n_0)!\displaystyle \prod_{\lambda\vdash k, \lambda\neq (1^k)}z_{x(\lambda)}z_\lambda^{l(x(\lambda))}}.
\end{eqnarray*}

Take $x$ and $y$ to be two proper families of partitions. For any integer $n>|x|,|y|$ we have the following equation in $Z(\mathbb{C}[\mathcal{B}_{kn}^k])$
\begin{equation}\label{eq_str_coe}
\mathbf{C}_x(n)\mathbf{C}_y(n)=\sum_{z}c_{xy}^h(n) \mathbf{C}_z(n),
\end{equation}
where $h$ runs through all proper families of partitions verifying $|z|\leq n.$

In \cite{Tout2017}, under some conditions, a formula describing the form of the structure coefficients of centers of finite group algebras is given. We show below that the sequence $(\mathcal{B}_{kn}^k)_n$ satisfies these conditions. This will allow us to use \cite[Corollary 6.3]{Tout2017} in order to give a polynomiality property for the structure coefficients $c_{xy}^h(n)$ described in Equation (\ref{eq_str_coe}). 

We will show below the conditions required in \cite{Tout2017} for our sequence of groups $(\mathcal{B}_{kn}^k)_n.$ To avoid repetitions and confusing notations and since the integer $k$ will be fixed, we will use simply $G_n$ to denote the group $\mathcal{B}_{kn}^k.$

\textbf{Hypothesis 1:} For any integer $1\leq r\leq n,$ there exists a group $G_n^r$ isomorphic to $G_{n-r}.$ Set 
$$G_n^r:=\lbrace \omega\in G_n \text{ such that } \omega(i)=i \text{ for any $1\leq i\leq kr$}\rbrace.$$
for this reason.

\textbf{Hypothesis 2:} The elements of $G_n^r$ and $G_r$ commute between each other which is normal since the permutations in these groups act on disjoint sets.

\textbf{Hypothesis 3:} $G_{n+1}^r\cap G_n=G_n^r$ which is obvious.

\textbf{Hypothesis 4:} For any $z\in G_n,$ $\mathrm{k}(G_n^{r_1} zG_n^{r_2}):=\min\lbrace s|G_n^{r_1} zG_n^{r_2}\cap G_s\neq \emptyset\rbrace\leq r_1+r_2.$ To prove this, remark first that the size of the set $\lbrace 1,\cdots,kr_1\rbrace \cap \lbrace z(1),\cdots,z(kr_2)\rbrace$ is a multiple of $k$ since $z\in G_n,$ say it is $km.$ Suppose that $$\lbrace h_1,\cdots, h_{kr_1-km}\rbrace =\lbrace1,\cdots,kr_1\rbrace\setminus \lbrace z(1),\cdots,z(kr_2)\rbrace.$$ We can find a permutation of the following form
$$\begin{matrix}
1 & 2 & \cdots & kr_2 & kr_2+1 & \cdots & kr_1+kr_2-km & kr_1+kr_2-km+1 & \cdots & kn \\
z(1) & z(2) & \cdots & z(kr_2) & h_1 & \cdots & h_{kr_1-km} & * & \cdots & *
\end{matrix}$$
in $zG_n^{r_2}$ since it contains permutations that fixes the first $kr_2$ images of $z.$ The stars are used to say that the images may not be fixed. Since the multiplication by an element of $G_n^{r_1}$ to the left permutes the elements greater than $kr_1$ in the second line defining this permutation, the set $G_n^{r_1} zG_n^{r_2}$ contains thus a permutation of the following form
$$\begin{matrix}
1 & 2 & \cdots & kr_2 & kr_2+1 & \cdots & kr_1+kr_2-km & kr_1+kr_2-km+1 & \cdots & kn \\
* & * & \cdots & * & h_1 & \cdots & h_{kr_1-km} & kr_1+kr_2-km+1 & \cdots & kn
\end{matrix}$$
This permutation is also in $G_{r_1+r_2-m}$ which ends the proof. 

\textbf{Hypothesis 5:} If $z\in G_n$ then we have $zG_n^{r_1}z^{-1}\cap G_n^{r_2}= G_n^{r(z)}$ where 
\begin{eqnarray*}
r(z)&=&|\lbrace z(1),z(2),\cdots , z(kr_1),1,\cdots ,kr_2\rbrace| \\
&=&kr_1+kr_2-|\lbrace z(1),z(2),\cdots , z(kr_1)\rbrace\cap\lbrace 1,\cdots ,kr_2\rbrace|.
\end{eqnarray*}
To prove this, let $a=zbz^{-1}$ be an element of $G_n$ which fixes the $kr_2$ first elements while $b$ fixes the $kr_1$ first elements. Then $a$ also fixes the elements $z(1),\cdots, z(kr_1)$ which proves that $zG_n^{r_1}z^{-1}\cap G_n^{r_2}\subset G_n^{r(z)}.$ In the opposite direction, if $p$ is a permutation of $n$ which fixes the elements of the set $\lbrace z(1),z(2),\cdots , z(kr_1),1,\cdots ,kr_2\rbrace$ then $p$ is in $G_n^{r_2}$ and in addition $z^{-1}pz$ is in $G_n^{r_1}$ which implies that $p=zz^{-1}pzz^{-1}$ is in $zG_n^{r_1}z^{-1}.$\\

Now with all the necessary hypotheses verified we can apply the main result in \cite{Tout2017} to get the following theorem.

\begin{theoreme}\label{main_the}
Let $x,$ $y$ and $h$ be three proper families of partitions. For any $n>|x|,|y|,|h|,$ the coefficients $c_{xy}^h(n)$ defined in Equation \ref{eq_str_coe} are polynomials in $n$ with non-negative rational coefficients. In addition,
$$\deg(c_{xy}^h(n))<|x|+|y|-|h|.$$ 
\end{theoreme}
\begin{proof}
By \cite[Corollary 6.3]{Tout2017}, if $n>|x|,|y|,|h|$ then 
\begin{equation*}
c_{xy}^h(n)=\frac{|C_x(n)||C_y(n)||\mathcal{B}_{k(n-|x|)}^k||\mathcal{B}_{k(n-|y|)}^k|}{|\mathcal{B}_{kn}^k||C_h(n)|}\sum_{|h|\leq r \leq |x|+|y|}\frac{a_{xy}^h(r)}{|\mathcal{B}_{k(n-r)}^k|}
\end{equation*}
where the $a_{xy}^h(r)$ are positive, rational and independent numbers of $n.$ Since all the cardinals involved in this formula are known, we get after simplification the following formula for $c_{xy}^h(n)$
\begin{equation*}
c_{xy}^h(n)=(k!)^{l(h(1^k))-l(x(1^k))-l(y(1^k))}\frac{z_{h(1^k)}c_h}{z_{x(1^k)}z_{y(1^k)}c_xc_y}\sum_{|h|\leq r \leq |x|+|y|}\frac{(k!)^{r-|h|}a_{xy}^h(r)(n-|h|)!}{(n-r)!}
\end{equation*}
where $c_x$ denotes $\prod_{\lambda\vdash k, \lambda\neq (1^k)}z_{x(\lambda)}z_\lambda^{l(x(\lambda))}.$ The result follows.
\end{proof}

\subsection{Special cases} In this section we revisit the two already known results of polynomiality for the structure coefficients of the center of the symmetric group ($k=1$) algebra and the center of the hyperoctahedral ($k=2$) group algebra. In addition, as an application of our main theorem, we give a polynomiality property in the case $k=3.$ In these three cases, we give explicit expressions of products of conjugacy classes in the associated center algebra in order to see our results.  

\subsubsection{k=1: The symmetric group} As seen in Section \ref{sec_2}, the conjugacy classes of the symmetric group $\mathcal{S}_n$ are indexed by partitions of $n.$ If $\lambda$ is a partition of $n$ the size of its associated conjugacy class is 
$$|C_\lambda|=\frac{n!}{z_\lambda}.$$

If $\lambda$ is a proper partition with $|\lambda|<n,$ we define $\underline{\lambda}_n$ to be the partition $\lambda\cup (1^{n-|\lambda|}).$ Now let $\lambda$ and $\delta$ be two proper partitions with $|\lambda|,|\delta|<n.$ In the center of the symmetric group algebra we have the following equation:
\begin{equation}\label{Eq_str_sym}
\textbf{C}_{\underline{\lambda}_n}\textbf{C}_{\underline{\delta}_n}=\sum_{\gamma}c_{\lambda\delta}^\gamma (n)\textbf{C}_{\underline{\gamma}_n}
\end{equation}
where the sum runs through all proper partitions $\gamma$ satisfying $|\gamma|\leq |\lambda|+|\delta|.$ If we apply Theorem \ref{main_the}, we re-obtain the following result of Farahat and Higman in \cite[Theorem 2.2]{FaharatHigman1959}.

\begin{theoreme}
Let $\lambda,$ $\delta$ and $\gamma$ be three proper partitions and let $n\geq |\lambda|,|\delta|,|\gamma|$ be an integer. The structure coefficient $c_{\lambda\delta}^{\gamma}(n)$ of the center of the symmetric group algebra defined by Equation (\ref{Eq_str_sym}) is a polynomial in $n$ with non-negative coefficients and 
$$\deg(c_{\lambda\delta}^{\gamma}(n))\leq |\lambda|+|\delta|-|\gamma|.$$
\end{theoreme}

\begin{ex}
See \cite{touPhd14} for more information about the computation of the following two complete expressions in the center of the symmetric group $\mathcal{S}_n,$
$$\mathbf{C}_{(1^{n-2},2)}\mathbf{C}_{(1^{n-2},2)}=\frac{n(n-1)}{2} \mathbf{C}_{(1^{n})}+3\mathbf{C}_{(1^{n-3},3)}+2\mathbf{C}_{(1^{n-4},2^2)} \text{ for any $n\geq 4$},$$
and 
$$\mathbf{C}_{(1^{n-2},2)}\mathbf{C}_{(1^{n-3},3)}=2(n-2) \mathbf{C}_{(1^{n-2},2)}+4\mathbf{C}_{(1^{n-4},4)}+\mathbf{C}_{(1^{n-5},2,3)} \text{ for any $n\geq 5$}.$$
\end{ex}

\subsubsection{k=2: The hyperoctahedral group} In Section \ref{sec_hyp}, we showed that the conjugacy classes of the hyperoctahedral group are indexed by pairs of partitions $(\lambda,\delta)$ such that $|\lambda|+|\delta|=n.$ The partition $\lambda$ is that associated to the partition $(1,1)$ of $2$ while $\delta$ is associated to the partition $(2).$ The size of the class $C_{(\lambda,\delta)}$ is given in Corollary \ref{size_conj_k=2}

$$|C_{(\lambda,\delta)}|=\frac{2^nn!}{2^{l(\lambda)+l(\delta)}z_{\lambda}z_{\delta}}.$$

By Definition \ref{def_propre}, the pair $(\lambda,\delta)$ is proper if and only if the partition $\lambda$ is proper. For a proper pair $(\lambda,\delta)$ of partitions and for any integer $n>|\lambda|+|\delta|,$ we define the following pair of partitions:
$$\underline{(\lambda,\delta)}_n:=(\lambda\cup (1^{n-|\lambda|-|\delta|}),\delta).$$ Let $(\lambda_1,\delta_1)$ and $(\lambda_2,\delta_2)$ be two proper pairs of partitions. We have the following equation in the center of the hyperoctahedral group algebra for any integer $n$ greater than $|\lambda_1|+|\delta_1|, |\lambda_2|+|\delta_2|,$

\begin{equation}\label{Eq_str_hyp}
\textbf{C}_{\underline{(\lambda_1,\delta_1)}_n}\textbf{C}_{\underline{(\lambda_2,\delta_2)}_n}=\sum_{(\lambda_3,\delta_3)}c_{(\lambda_1,\delta_1)(\lambda_2,\delta_2)}^{(\lambda_3,\delta_3)} (n)\textbf{C}_{\underline{(\lambda_3,\delta_3)}_n}
\end{equation}
where the sum runs over all the proper pairs of partitions $(\lambda_3,\delta_3)$ satisfying $|\lambda_3|+|\delta_3|\leq |\lambda_1|+|\delta_1|+\lambda_2|+|\delta_2|.$ As an application of Theorem \ref{main_the}, we re-obtain the following result in \cite[Corollary 6.11]{Tout2017}.

\begin{cor}
Let $(\lambda_1,\delta_1), (\lambda_2,\delta_2)$ and $(\lambda_3,\delta_3)$ be three proper pairs of partitions, then for any $n\geq |\lambda_1|+|\delta_1|,|\lambda_2|+|\delta_2|,|\lambda_3|+|\delta_3|$ the structure coefficient $c_{(\lambda_1,\delta_1)(\lambda_2,\delta_2)}^{(\lambda_3,\delta_3)}(n)$ of the center of the hyperoctahedral group algebra defined in Equation (\ref{Eq_str_hyp}) is a polynomial in $n$ with non-negative coefficients and we have
$$\deg(c_{(\lambda_1,\delta_1)(\lambda_2,\delta_2)}^{(\lambda_3,\delta_3)}(n))\leq |\lambda_1|+|\delta_1|+|\lambda_2|+|\delta_2|-|\lambda_3|-|\delta_3|.$$
\end{cor}

\begin{ex}
We give in this example the complete product of the class $C_{((1^{n-2}),(2))}$ by itself whenever $n\geq 4:$
$$\mathbf{C}_{((1^{n-2}),(2))}\mathbf{C}_{((1^{n-2}),(2))}=n(n-1)\mathbf{C}_{((1^{n}),\emptyset)}+2\mathbf{C}_{((1^{n-4}),(2^2))}+2\mathbf{C}_{((1^{n-2}),(1^2))}+3\mathbf{C}_{((1^{n-3},3),\emptyset)}.$$

${C}_{((1^{n}),\emptyset)}$ is the identity class and since any element in ${C}_{((1^{n-2}),(2))}$ has its inverse in ${C}_{((1^{n-2}),(2))},$ the coefficient of $\mathbf{C}_{((1^{n}),\emptyset)}$ is the size of the conjugacy class ${C}_{((1^{n-2}),(2))}$ which is $n(n-1).$ The coefficient of $\mathbf{C}_{((1^{n-4}),(2^2))}$ is $2$ since if we fix a permutation of ${C}_{((1^{n-4}),(2^2))},$ say $$(1,3,2,4)(5,7,6,8)(9)(10)\cdots (2n),$$ 
then there exists only two pairs $(\alpha;\beta)\in {C}_{((1^{n-2}),(2))}\times {C}_{((1^{n-2}),(2))}$ such that $$\alpha\beta=(1,3,2,4)(5,7,6,8)(9)(10)\cdots (2n).$$ Mainly:
$$(\alpha,\beta)=((1,3,2,4)(5)\cdots (2n);(1)(2)(3)(4)(5,7,6,8)(9)\cdots (2n))$$
or 
$$(\alpha,\beta)=((1)(2)(3)(4)(5,7,6,8)(9)\cdots (2n);(1,3,2,4)(5)\cdots (2n)).$$ 
There are only two permutations in $\mathcal{B}_4^2,$ namely $\alpha=(1,3,2,4)$ and $\beta=(1,4,2,3),$ such that $\alpha,\beta\in {C}_{(\emptyset,(2))}$ and $\alpha^2=\beta^2=(12)(34).$ Thus the coefficient of $\mathbf{C}_{((1^{n-2}),(1^2))}$ is $2.$ The last coefficient can be obtained by identifying both sides.
\end{ex}

\subsubsection{k=3: the group $\mathcal{B}_{3k}^3$} Moving to the case $k=3,$ we showed in Section \ref{seck=3} that the conjugacy classes of the group $\mathcal{B}_{3n}^3$ are indexed by triplet of partitions $(\alpha,\beta,\gamma)$ such that $|\alpha|+|\beta|+|\gamma|=n.$ For a given triplet $(\alpha,\beta,\gamma),$ the size of its associated conjugacy class is given in Corollary \ref{size_con_k=3},

$$|C_{(\alpha,\beta,\gamma)}|=2^{n-l(\alpha)-l(\beta)}3^{n-l(\alpha)-l(\gamma)}.\frac{n!}{z_{\alpha}z_{\beta}z_{\gamma}}.$$

By Definition \ref{def_propre}, the triplet $(\alpha,\beta,\gamma)$ is proper if and only if the partition $\alpha$ is proper. Fix three proper triplet of partitions, $(\alpha_1,\beta_1,\gamma_1),$ $(\alpha_2,\beta_2,\gamma_2)$ and $(\alpha_3,\beta_3,\gamma_3),$ the structure coefficient associated to these three triplet has the following form according to the proof of Theorem \ref{main_the}, 

$$
c_{(\alpha_1,\beta_1,\gamma_1)(\alpha_2,\beta_2,\gamma_2)}^{(\alpha_3,\beta_3,\gamma_3)}(n)=(3!)^{l(\alpha_3)-l(\alpha_1)-l(\alpha_2)}\frac{z_{\alpha_3}z_{\beta_3}2^{l(\beta_3)}z_{\gamma_3}3^{l(\gamma_3)}}{z_{\alpha_1}z_{\alpha_2}z_{\beta_1}2^{l(\beta_1)}z_{\gamma_1}3^{l(\gamma_1)}z_{\beta_2}2^{l(\beta_2)}z_{\gamma_2}3^{l(\gamma_2)}}\times$$
$$\sum_{r}\frac{(3!)^{r-|\alpha_3|-|\beta_3|-|\gamma_3|}a_{(\alpha_1,\beta_1,\gamma_1)(\alpha_2,\beta_2,\gamma_2)}^{(\alpha_3,\beta_3,\gamma_3)}(r)(n-|\alpha_3|-|\beta_3|-|\gamma_3|)!}{(n-r)!}
$$
where the sum runs through all integers $r$ with $$|\alpha_3|+|\beta_3|+|\gamma_3|\leq r \leq |\alpha_1|+|\beta_1|+|\gamma_1|+|\alpha_2|+|\beta_2|+|\gamma_2|.$$

\begin{cor}
Let $(\alpha_1,\beta_1,\gamma_1),$ $(\alpha_2,\beta_2,\gamma_2)$ and $(\alpha_3,\beta_3,\gamma_3)$ be three proper triplet of partitions, then for any $n\geq |\alpha_1|+|\beta_1|+|\gamma_1|,|\alpha_2|+|\beta_2|+|\gamma_2|,|\alpha_3|+|\beta_3|+|\gamma_3|$ the structure coefficient $c_{(\alpha_1,\beta_1,\gamma_1)(\alpha_2,\beta_2,\gamma_2)}^{(\alpha_3,\beta_3,\gamma_3)}(n)$ of the center of the group $\mathcal{B}_{3n}^3$ algebra is a polynomial in $n$ with non-negative coefficients and we have
$$\deg(c_{(\alpha_1,\beta_1,\gamma_1)(\alpha_2,\beta_2,\gamma_2)}^{(\alpha_3,\beta_3,\gamma_3)}(n))\leq |\alpha_1|+|\beta_1|+|\gamma_1|+|\alpha_2|+|\beta_2|+|\gamma_2|-|\alpha_3|-|\beta_3|-|\gamma_3|.$$
\end{cor}

\begin{ex} For $n\geq 3,$ we leave it to the reader to verify the following two complete expressions in $Z(\mathbb{C}[\mathcal{B}_{3n}^3]):$

$$\mathbf{C}_{((1^{n-2}),(1),(1))}\mathbf{C}_{((1^{n-1}),\emptyset,(1))}=2\mathbf{C}_{((1^{n-3}),(1),(1^2))}+2(n-1)\mathbf{C}_{((1^{n-1}),(1),\emptyset)}+3\mathbf{C}_{((1^{n-2}),(1),(1))}$$
and
$$\mathbf{C}_{((1^{n-2}),(1),(1))}\mathbf{C}_{((1^{n-1}),(1),\emptyset)}=2\mathbf{C}_{((1^{n-3}),(1^2),(1))}+3(n-1)\mathbf{C}_{((1^{n-1}),\emptyset,(1))}+4\mathbf{C}_{((1^{n-2}),(1^2),\emptyset)}$$$$+6\mathbf{C}_{((1^{n-2}),\emptyset,(1^2))}.$$

\end{ex}

\bibliographystyle{plain}
\bibliography{biblio}
\label{sec:biblio}
\end{document}